\documentstyle[amssymb,amsfonts,12pt]{amsart}
\newtheorem{theorem}{Theorem}[section]

\newtheorem{proposition}[theorem]{Proposition}
\newtheorem{definition}[theorem]{Definition}

\theoremstyle{remark}
\newtheorem{remark}[theorem]{Remark}

\def\QSet{\mbox{\rm\kern.24em
\vrule width.03em height1.48ex depth-.051ex \kern-.26em Q}}

\def\E{{\mbox{\rm I\kern-.22em E}}}
\def\P{{\bf P}}

\def\T{{\bf T}}
\def\S{{\bf S}}
\def\Z{{\bf Z}}
\def\R{{\bf R}}
\def\N{{\bf N}}
\def\C{{\bf C}}
\def\supp{{\operatorname{supp}}}
\def\size{{\operatorname{size}}}

\def\P{{\mathcal P}}

\def\M{{\operatorname{M}}}

\def\F{{\mathcal F}}
\def\L{{\mathcal L}}

\def\G{{\mathcal G}}

\def\bas{\begin{align*}}
\def\eas{\end{align*}}
\def\bi{\begin{itemize}}
\def\ei{\end{itemize}}
\newenvironment{proof}{\noindent {\bf Proof} }{\endprf\par}
\def \endprf{\hfill  {\vrule height6pt width6pt depth0pt}\medskip}
\def\emph#1{{\it #1}}

\begin{document}
\title{Improved range in the Return Times Theorem}

\author{Ciprian Demeter}
\address{Department of Mathematics, Indiana University, 831 East 3rd St., Bloomington IN 47405}
\email{demeterc@@indiana.edu}

\thanks{The author was supported by a Sloan Research Fellowship and by NSF Grant DMS-0556389}
\begin{abstract}
We prove that the  Return Times Theorem holds true for pairs of $L^p-L^q$ functions, whenever $\frac{1}{p}+\frac{1}{q}<\frac{3}{2}$
\end{abstract}

\maketitle
\section{Introduction}
Let ${\bf X}=(X,\Sigma,\mu, \tau)$ be a dynamical system, that is a Lebesgue space $(X,\Sigma,\mu)$ equipped with an invertible bimeasurable  measure preserving transformation $\tau:X\to X$. We recall that a complete probability space $(X,\Sigma,\mu)$ is called a \emph{Lebesgue space} if it is isomorphic with the ordinary Lebesgue measure space $([0,1),\L,m)$, where $\L$ and $m$ denote the usual Lebesgue algebra and measure (see \cite{Ha} for more on this topic).  The system ${\bf X}$ is called \emph{ergodic} if $A\in\Sigma$ and $\mu(A\vartriangle\tau^{-1}A)=0$ imply $\mu(A)\in\{0,1\}.$

In \cite{Bo5} Bourgain proved the following result.

\begin{theorem}[Return times theorem]
\label{Bretthm}
For each function $f\in L^{\infty}(X)$ there is a universal set $X_0\subseteq X$ with $\mu(X_0)=1$, such that for each second dynamical system  ${\bf Y}=(Y,\F,\nu,\sigma)$, each $g\in L^{\infty}(Y)$ and each $x\in X_0$, the averages
$$\lim_{N \to \infty} \frac1{N}\sum_{n=1}^{N}f(\tau^nx)g(\sigma^ny)$$
converge $\nu$-almost everywhere.
\end{theorem}
Subsequent proofs were given in \cite{BFKO} and \cite{Ru}.
If in the above theorem  $f$ (or $g$) is taken to be a constant function, one recovers the classical Birkhoff's pointwise ergodic theorem, see \cite{Bi}. However, Theorem ~\ref{Bretthm} is much stronger, in that it shows that given $f$, for almost every $x$, the  sequence $w_n=(f(\tau^nx))_{n\in\N}$ forms a system of universal weights for the  pointwise ergodic theorem.

The difficulty in Theorem \ref{Bretthm} lies in the fact that the weights provided by $f$ work for {\em every} dynamical system ${\bf Y}=(Y,\F,\nu,\sigma)$. If on the other hand, the system ${\bf Y}=(Y,\F,\nu,\sigma)$ is fixed, then the result follows from an approximation argument combined with applications of Birkhoff's theorem to the functions $f\otimes g_j$ in the product system ${\bf X}\times{\bf Y}$, where $(g_j)_j$ is a dense class of functions in $L^2(Y)$.

A result by  Assani, Buczolich and   Mauldin \cite{ABM} shows that  the Return Times Theorem  fails when $p=q=1$:

\begin{theorem}\cite{ABM}
Let ${\bf X}=(X,\Sigma,\mu, \tau)$
 be an ergodic dynamical system. There exist a function $f\in L^1(X)$ and a subset $X_0\subseteq X$ of full  measure with the following property: for each $x_0\in X_0$ and for each ergodic dynamical system ${\bf Y}=(Y,\F,\nu,\sigma)$, there exists $g\in L^{1}(Y)$ such that the averages
$$\lim_{N \to \infty} \frac1N\sum_{n=1}^{N}f(\tau^nx_0)g(\sigma^ny)$$
diverge for almost every y.
\end{theorem}
On the other hand, H\"older's inequality and an elementary density argument show that Bourgain's theorem holds for $f\in L^{p}(X)$ and $g\in L^{q}(Y)$, whenever $1\le p,q\le \infty$ and $\frac1p+\frac1q\le 1$, see \cite{Ru}, or section 4 in \cite{DLTT}.
It is an interesting question to understand the precise range of $p$ and $q$ for which a positive result holds.

A significant progress on this issue appears in \cite{DLTT}, where it is proved that the Return Times Theorem remains valid when $q\ge 2$ and $p>1$. We build on the approach from \cite{DLTT} and prove:
\begin{theorem}
\label{Bretthmext}
Let $1<p,q\le \infty$ be such that
\begin{equation}
\label{restrictionpandq}
\frac1p+\frac1q<\frac32
\end{equation}
For each dynamical system $(X,\Sigma,\mu, \tau)$ and each  $f\in L^{p}(X)$ there is a universal set $X_0\subseteq X$ with $\mu(X_0)=1$, such that for each second dynamical system  ${\bf Y}=(Y,\F,\nu,\sigma)$, each $g\in L^{q}(Y)$ and each $x\in X_0$, the averages
$$\lim_{N \to \infty} \frac1{N}\sum_{n=1}^{N}f(\tau^nx)g(\sigma^ny)$$
converge $\nu$-almost everywhere.
\end{theorem}

%%%%%%%%%%%%%%%%%%%%%%%%%%%%%%%%%%
Given the result in \cite{DLTT} and the convergence for $L^{\infty}$ functions $f$ and $g$, an approximation argument  will immediately prove Theorem \ref{Bretthmext}, once we establish the following maximal inequality:

\begin{theorem}
\label{otnumber1}Let $1<p<\infty$ and  $1<q<2$ satisfy \eqref{restrictionpandq}.
For each dynamical system ${\bf X}=(X,\Sigma,\mu, \tau)$ and each $f\in L^p(X)$
\begin{equation}
\label{hgdhgdhgh769878}
\|\sup_{(Y,\F,\nu,\sigma)}\sup_{\|g\|_{L^q(Y)=1}}\|\sup_{N}|\frac1{N}\sum_{n=1}^{N}f(\tau^nx)g(\sigma^ny)|\|_{L^q_y(Y)}\|_{L^p_x(X)}\lesssim_{p,q} \|f\|_{L^p(X)},
\end{equation}
where the first supremum in the inequality above is taken over all dynamical systems ${\bf Y}=(Y,\F,\nu,\sigma)$.  Here and in the following, we have subscripted some of our $L^p$ norms to clarify the variable being integrated over.
\end{theorem}

As explained in \cite{DLTT}, this will follow by standard transfer, from the following real line version.

\begin{theorem}
\label{t.returntimeave}Let $1<p<\infty$ and  $1<q<2$ satisfy \eqref{restrictionpandq}.
For each $f\in L^p(\R)$ we have
\begin{equation}
\label{e.returntimeave}
\left\|\sup_{\| g\|_{L^q(\R)}=1}\|\sup_{k\in\Z}\frac{1}{2^{k+1}} \int_{-2^k}^{2^k}|f(x+y)g(z+y)|dy\|_{L^q_z(\R)}\right\|_{L^p_x(\R)}\lesssim_{p,q} \|f\|_{L^p(\R)}.
\end{equation}
\end{theorem}

When $1\le p,q\le \infty$ are in the duality range, that is when $\frac{1}{p}+\frac{1}{q}\le 1$,  Theorem \ref{t.returntimeave} follows immediately from H\"older's inequality. The case $q=2$, $1<p<\infty$ was proved in \cite{DLTT}. The approach from \cite{DLTT} consists of treating averages and singular integrals in a similar way: one performs Littlewood-Paley decompositions of each average, combined with Gabor frames expansions of $f$, to obtain a model sum. This discretized operator turns out to be a maximal truncation of the Carleson operator \cite{Car}
$$Cf(x,\theta)=p.v.\int_{\R}f(x+y)e^{iy\theta}\frac{dy}{y}.$$
The analysis in \cite{DLTT} is then driven by time-frequency techniques combined with an $L^2$ maximal multiplier result of Bourgain. Most of the work in \cite{DLTT} is $L^2$ based, and in particular, the fact that $q=2$ in Theorem \ref{t.returntimeave} is heavily exploited.

In this paper, we relax the restriction $q=2$, and replace it with \eqref{restrictionpandq}. There are two key new ingredients. The first is a simplification of the argument from \cite{DLTT}, which consists of treating the Hardy-Littlewood kernel in a way distinct from the Hilbert kernel. In \cite{DLTT}, the two kernels were treated on equal footing, as a byproduct of a unified approach for regular averages and signed averages.
Here we treat each average as a single Littlewood-Paley piece. This decomposition  simplifies the model sum to a significant extent, and is suited for analysis on spaces other than $L^2$. The main new ingredient we use here is the $L^q$ version of Bourgain's result on maximal multipliers, that we proved in \cite{Dem} (see Theorem \ref{Blemma} below).

It is interesting to remark that the range \eqref{restrictionpandq} that we establish is the same as the range where the Bilinear Hilbert Transform (see \cite{LTBilH}, \cite{LT2})
$$BHT(f,g)(x)=p.v.\int_\R f(x+y)g(x-y)\frac{dy}{y}$$
and the Bilinear maximal function (see \cite{La})
$$BM(f,g)(x)=\sup_{t>0}|\frac{1}{2t}\int_{-t}^{t} f(x+y)g(x-y)dy|$$
are known to be bounded. This is perhaps not a coincidence, as the methods we use to prove Theorem \ref{t.returntimeave}
are related to those used in the proof of the Bilinear Hilbert Transform. Moreover, in both cases, the methods fail beyond the $3/2$ threshold, essentially because of the same reason. Even the model sum that contains a single scale is unbounded if $\frac1p+\frac1q\ge 3/2$. Another interesting connection is that
both the boundedness of the bilinear maximal function and the Return Times Theorem fail for pairs of $L^1$ functions, and they do so in quite a dramatic way.
Even the (smaller) tail operators
$$T_1(f,g)(x):=\sup_{t>1}|\frac{1}{2t}\int_{t}^{t+1} f(x+y)g(x-y)dy|$$
$$T_2(f,g)(x,y):=\sup_{n}|\frac1{n}f(\tau^nx)g(\sigma^ny)|$$
fail to be bounded for pairs of $L^1$ functions. See \cite{AB1}, and \cite{ABM}.

This project is a continuation of the work in \cite{DLTT}.  The author is indebted to M. Lacey, C. Thiele and T. Tao for helpful conversations over the last few years.

%%%%%%%%%%%%%%%%%%%%%%%%%%%
\section{Discretization}
Let $m_k:\R\to\R$ be a sequence  of multipliers. For each $1\le q\le \infty$, the maximal multiplier norm associated with them is defined as
$$\|(m_k)_{k \in \Z}\|_{M_q^{*}(\R)}
 :=\sup_{\|g\|_q=1}\left\|\sup_k|\int m_k(\theta)\widehat{g}(\theta)e^{2\pi i\theta z}d\theta|\right\|_{L^q_z(\R)}.$$

Let $K:\R\to [0,\infty)$ be a positive function with $K(0)>0$, whose Fourier transform is supported in -say- the interval $[-1,1]$. In particular, one can take $K$ to be the inverse Fourier transform of $\eta*\tilde{\eta}$, where $\eta:\R\to\R$ is supported in $[-1/2,1/2]$, $\int \eta\not=0$, and $\tilde{\eta}(\xi)=\eta(-\xi)$. Of course, Theorem \ref{t.returntimeave} will immediately follow if we can prove the same thing with $$\sup_{\| g\|_{L^q(\R)}=1}\|\sup_{k\in\Z}\frac{1}{2^{k+1}} \int_{-2^k}^{2^k}|f(x+y)g(z+y)|dy\|_{L^q_z(\R)}$$
replaced by
$$Rf(x):=\sup_{\| g\|_{L^q(\R)}=1}\|\sup_{k\in\Z}\frac{1}{2^{k}} |\int f(x+y)g(z+y)K(\frac{y}{2^k})dy|\|_{L^q_z(\R)}.$$
As remarked earlier, whenever $p\ge\frac{q}{q-1}$ we know that $R$ maps $L^p$ to $L^p$. By invoking restricted weak type interpolation, it thus suffices to prove that
\begin{equation}
\label{dssdsdfsdfdfiuiuoiuer3.r.>r[p][tpp348337rryfg}
m\{x:R1_F(x)>\lambda\}\lesssim_{p,q} \frac{|F|}{\lambda^p},
\end{equation}
for each $p<2$, $\lambda\le 1$ and each finite measure set $F\subset\R$.

We next indicate how to discretize the operator $R$. Rather than going in detail through the whole procedure, we emphasize its key aspects. The interested reader is referred to section 6 in \cite{DLTT} for details. We note however that our approach here is a simplified version of the decomposition in \cite{DLTT}, since we no longer perform Littlewood-Paley decompositions of a given average.

Let $\varphi$ be a Schwartz function such that $\widehat{\varphi}$ is supported in $[0,1]$ and such that
$$\sum_{l\in\Z}|\widehat{\varphi}(\xi-\frac{l}{2})|^2\equiv C.$$
If $C$ is chosen appropriately, it will follow that for each $F$ and each $k\in\Z$, one has the following Gabor basis expansion
$$\sum_{m,l\in\Z}\langle 1_F,\varphi_{k,m,l/2} \rangle\varphi_{k,m,l/2}=1_F.$$
Here
$$\varphi_{k,m,l}(x):=2^{-\frac{k}{2}}\varphi(2^{-k}x-m)e^{2\pi i2^{-k}xl},$$
is the $L^2$ normalized wave packet that is quasi-localized in time frequency in the rectangle $[m2^k,(m+1)2^k]\times [l2^{-k},(l+1)2^{-k}]$.

Given a scale $2^k$, one uses this expansion to get
$$R1_F(x)=\sup_{\| g\|_{L^q(\R)}=1}\|\sup_{k\in\Z}|\sum_{m,l\in\Z}\langle 1_F,\varphi_{k,m,l/2}\rangle\int\varphi_{k,m,l/2}(x+y)g(z+y)2^{-k}K(\frac{y}{2^k})dy|\|_{L^q_z(\R)}=$$
$$=\left\|\left(\sum_{m,l\in\Z}\langle 1_F,\varphi_{k,m,l/2}\rangle\F[\varphi_{k,m,l/2}(x+\cdot)2^{-k}K(\frac{\cdot}{2^k})](\theta)\right)_{k\in\Z}\right\|_{M_{q,\theta}^*(\theta)}.$$

The key observation is that the function
$$\phi_{k,m,l/2}(x,\theta)=\F[\varphi_{k,m,l/2}(x+\cdot)2^{-k}K(\frac{\cdot}{2^k})](\theta),$$
has the same decay (in $x$) as $\varphi_{k,m,l/2}$, and behaves like the function $\varphi_{k,m,l/2}(x)1_{[l2^{-k},(l+1)2^{-k}]}(\theta)$. Note that in reality, the support in $\theta$ of $\phi_{k,m,l/2}(x,\theta)$ is slightly larger than $[l2^{-k},(l+1)2^{-k}]$, more precisely it is a subset of $[l2^{-k},(l+1)2^{-k}]+[-2^{-k},2^{-k}]$. This will force upon us the use of shifted dyadic grids. But, as explained in \cite{DLTT}, for simplicity of notation (but not of the argument), we can really assume (and will do so) that we are working with the standard dyadic grid.

We will denote by $\S_{univ}$ the collection of all tiles $s=I_s\times \omega_s$ with area 1, where both $I_s$ and $\omega_s$ are dyadic intervals. We will refer to $I_s$, $\omega_s$ as the time and frequency components of $s$.

\begin{definition}
A collection $\S\subset \S_{univ}$ of tiles will be referred to as convex, if whenever $s,s''\in\S$ and $s'\in\S_{univ}$, $ \omega_{s''}\subseteq\omega_{s'}\subseteq\omega_{s}$ and  $I_s\subseteq I_{s'}\subseteq I_{s''}$ will imply that $s'\in\S$.
\end{definition}
The fact that we choose to work with convex collections of tiles is of technical nature. It will allow us to use some results -like Proposition \ref{BMOtree}- which are known to hold under the convexity assumption.

As explained in \cite{DLTT}, \eqref{dssdsdfsdfdfiuiuoiuer3.r.>r[p][tpp348337rryfg} now follows from the following theorem

\begin{theorem}
\label{ceamaiceahh}
Let $1<q<2$. Let $\S$ be some arbitrary convex finite collection of tiles.
Consider also two  collections $\{\phi_s,s\in\S\}$ and $\{\varphi_s,s\in\S\}$ of Schwartz functions. We assume the functions $\phi_s:\R^2\to\R$ satisfy
\begin{equation}
\label{xbfgr1456wd}
\supp_{\theta}(\phi_{s}(x,\theta))\subseteq \omega_{s},\;\hbox{for each\;}x
\end{equation}
\begin{equation}
\label{xbfgr1456jr}
\supp_{\xi}(\F_x(\phi_{s}(x,\theta))(\xi))\subseteq \omega_{s},\;\hbox{for each\;}\theta
\end{equation}
\begin{equation}
\label{xbfgr1456bb}
\sup_{c\in\omega_s}\left\|\frac{\partial^n}{\partial\theta^n}\frac{\partial^m}{\partial x^m}\left[\phi_{s}(x,\theta)e^{-2\pi icx}\right]\right\|_{L^{\infty}_\theta(\R)}\lesssim_{n,m,M} |I_s|^{(n-m-1/2)}\chi_{I_s}^M(x),\;\; \forall n,m,M\ge 0,
\end{equation}
uniformly in $s$.
We also assume that the functions $\varphi_s:\R\to\R$ satisfy
\begin{equation}
\label{xbfgr1456wd731q}
\supp (\widehat{\varphi_{s}})\subseteq \omega_{s}
\end{equation}
and
\begin{equation}
\label{xbfgr14561o}
\sup_{c\in\omega_s}\left|\frac{\partial^n}{\partial x^n}\left[\varphi_{s}(x)e^{-2\pi icx}\right]\right|\lesssim_{n,M} {|I_s|}^{-n-\frac12}\chi_{I_s}^M(x),\;\; \forall n,M\ge 0
\end{equation}
uniformly in $s$.

Then the following inequality holds for each measurable $F\subset\R$ with finite measure, each $0<\lambda\le 1$, and each  $1<p<2$  such that $\frac{1}{p}+\frac1q<\frac32$
\begin{equation}
\label{en.511}
m\{x:\|(\sum_{s\in \S\atop{|I_s|=2^k}}\langle 1_F,\varphi_s\rangle\phi_s(x,\theta))_{k\in\Z}\|_{M_{q,\theta}^*(\R)}>\lambda\}\lesssim \frac{|F|}{\lambda^p},
\end{equation}
The implicit constant depends only on $p,q$ and on the implicit constants in  ~\eqref{xbfgr1456bb} and  ~\eqref{xbfgr14561o} (in particular, it is independent of $\S$, $F$ and $\lambda$).
\end{theorem}

The rest of the paper is devoted to proving this theorem. We fix the collection $\S$ throughout the rest of the paper.

%%%%%%%%%%%%%%%%%%%%%%%%%%%%%%%%%%%%%%%%%%%%%%%%%%%
\section{Trees}
\label{sec:trees}

We now recall some facts about trees. We refer the reader to \cite{Thiel2}, \cite{LTBilH} and \cite{MTT1} for more details.

\begin{definition}[Tile order]
For two tiles $s$ and $s'$ we write $s\le s'$ if $I_s\subseteq I_{s'}$ and $\omega_{s'}\subseteq\omega_s$.
\end{definition}

\begin{definition}[Trees]
A \emph{tree} with \emph{top} $(I_{\T},\xi_{\T})$, where $I_{\T}$ is an arbitrary (not necessarily dyadic) interval and $\xi_{\T}\in\R$, is a convex collection of tiles $\T\subseteq\S$ such that
$I_s\subset I_{\T}$ and $\xi_{\T}\in\omega_s$ for each $s\in \T$.

We will say that the tree has top tile $T\in \T$ if $s\le T$ for each $s\in \T$.
\end{definition}

\begin{remark}
Not all trees have a top tile, but of course, each tree can be (uniquely) decomposed into a disjoint union of trees with top tiles, such that these top tiles are pairwise disjoint.

Note also each tree $\T$ with top tile $T$ can be regarded as a tree with top $(I,\xi)$, for each interval $I_T\subseteq I$ and each  $\xi\in \omega_T$. If this is the case, we will adopt the convention that $I_{\T}:=I_T$.
\end{remark}

We now recall a few definitions and results from \cite{Dem}.
We will denote by $T_m$ the Fourier quasi-projection associated with the multiplier $m$:
$$T_mf(x):=\int \widehat{f}(\xi)m(\xi)e^{2\pi i\xi x}d\xi.$$

We will use the notation
$$\tilde{\chi}_{I}(x)=(1+\frac{|x-c(I)|}{|I|})^{-1}.$$

\begin{definition}
\label{adsfdfdwqryetluidhweilh}
Let $f$ be a $L^2$ function  and let $\S'\subset\S$.
We define the size $\size(\S')$ of $\S'$ relative\footnote{The function with respect to which the size is computed will change throughout the paper; however, it will always be clear from the context} to $f$ as
$$\size(\S'):=\sup_{s\in \S'}\sup_{m_s}\frac{1}{|I_s|^{1/2}}\|\tilde{\chi}_{I_s}^{10}(x)T_{m_s}f(x)\|_{L^2_{x}},$$
where  $m_s$ ranges over all functions adapted to $10\omega_s$.
\end{definition}

Each tree defines a region in the time-frequency plane. A good heuristic for the size of the tree is to think of it as being comparable to the  $L^{\infty}$ norm of the restriction of $f$ to this region. This heuristic is made precise by means of the phase space projections. We refer the reader to \cite{MTT8} and \cite{Dem} for more details.

We recall two important results regarding the size. The first one is immediate.
\begin{proposition}
\label{g11a00uugg55}
For each $\S'\subset\S$ and each $f\in L^1(\R)$ we have
$$\size(\S')\lesssim \sup_{s\in\S'}\inf_{x\in I_s}M f(x),$$
where the size is understood with respect to $f$.
\end{proposition}

The following Bessel type inequality from \cite{LTCar} will be useful in organizing  collections of tiles into trees.
\begin{proposition}
\label{Besselsineq}
Let $\S'\subseteq\S$ be a convex collection  of tiles and define $\Delta:=[-\log_2(\size(\S'))]$, where the size is understood with respect to some function $f\in L^2(\R)$. Then $\S'$ can be written as a disjoint union $\S'=\bigcup_{n\ge \Delta}\P_n,$ where $\size(\P_n)\le 2^{-n}$ and each $\P_n$ is convex and consists of a family $\F_{\P_n}$ of pairwise disjoint trees (that is, distinct trees do not share tiles) $\T$ with top tiles $T$ satisfying
\begin{equation}
\label{e:intp}
\sum_{\T\in\F_{\P_n}}|I_T|\lesssim 2^{2n}\|f\|_2^2,
\end{equation}
with bounds independent of $\S'$, $n$ and $f$.
\end{proposition}

We next recall an important decomposition from \cite{DLTT}.
Let $\T$ be a tree with top $(I_{\T},\xi_{\T})$.
For each $s\in\T$ and scale $l\ge 0$ we split  $\phi_s(x,\theta)$ as $$\phi_s(x,\theta)=\tilde{\phi}_{s,\T}^{(l)}(x,\theta)+\phi_{s,\T}^{(l)}(x,\theta).$$ For convenience, we set $\phi_{s,\T}^{(0)}:=\phi_{s}$ for each $s\in\T$. For $l\ge 1$ we define the first  piece to be localized in time:
 $$\operatorname {supp}\tilde{\phi}_{s,\T}^{(l)}(\cdot,\theta)\subseteq 2^{l-1}I_s,\;\;\hbox{for each\;}\theta\in\R.$$
For the second piece we need some degree of frequency localization, but obviously full localization as in the case of $\phi_s$ is impossible. We will content ourselves with preserving the mean zero property with respect to the top of the tree. The advantage of $\phi_{s,\T}^{(l)}$ over $\phi_s$ is that it gains extra decay in $x$. More precisely, we have for each $s\in \T$ and each $M\ge 0$

\begin{equation}
\label{e.m_s-supportati}
\phi_{s,\T}^{(l)}(x,\theta)e^{-2\pi i\xi_{\T}x}\;\hbox{ has mean zero},\;\;\theta\in\R,
\end{equation}

\begin{equation}
\label{e.m_s-size}
\phi_{s,\T}^{(l)}(x,\theta)e^{-2\pi i\xi_{\T}x}\hbox{\;is\;} c(M)2^{-Ml}-\hbox{adapted to}\; I_s,\;\;\hbox{for some constant\;}c(M), \;\;\theta\in\R,
\end{equation}

\begin{equation}
\label{e.m_s-supportgrr}
\operatorname {supp}\phi_{s,\T}^{(l)}(x,\cdot)\subset \omega_{s,2},\hbox{\;for each\;}x\in\R,
\end{equation}

 \begin{equation}
 \label{e.m_s-smoothgrr}
| \tfrac{ d}{d \theta} \phi_{s,\T}^{(l)}(x,\theta) |\lesssim 2^{-Ml} |I_s|^{\frac12} \chi_{I_s}^{M}(x),\;\;\hbox{uniformly in \;}x,\theta\in\R.
\end{equation}

We achieve this decomposition by first choosing a smooth function $\eta$ such that $\supp(\eta)\subset [-1/2,1/2]$ and $\eta=1$ on $[-1/4,1/4]$. We then define
$$\tilde{\phi}_{s,\T}^{(l)}(\theta;x):=\phi_{s}(\theta;x)\eta\text{Dil}_{2^{l}I_s}^{\infty}\eta(x)-\frac{e^{2\pi i\xi_{\T}x}\text{Dil}_{2^{l}I_s}^{\infty}\eta(x)}{\int_{\mathbb R}\text{Dil}_{2^{l}I_s}^{\infty}\eta(x)dx}\int_{\mathbb R}\phi_{s}(\theta;x)e^{-2\pi i\xi_{\T}x}\text{Dil}_{2^{l}I_s}^{\infty}\eta(x)dx$$
and
$$\phi_{s,\T}^{(l)}(\theta;x):=\frac{e^{2\pi i\xi_{\T}x}\text{Dil}_{2^{l}I_s}^{\infty}\eta(x)}{\int_{\mathbb R}\text{Dil}_{2^{l}I_s}^{\infty}\eta(x)dx}\int_{\mathbb R}\phi_{s}(\theta;x)e^{-2\pi i\xi_{\T}x}\text{Dil}_{2^{l}I_s}^{\infty}\eta(x)dx+\phi_{s}(\theta;x)(1-\text{Dil}_{2^{l}I_s}^{\infty}\eta(x)).$$
Properties ~\eqref{e.m_s-supportati} through ~\eqref{e.m_s-smoothgrr} are now easy consequences of ~\eqref{xbfgr1456wd}, ~\eqref{xbfgr1456jr} and ~\eqref{xbfgr1456bb}.

The following result is essentially Proposition 4.9 from \cite{Dem}. It can also be regarded as the "overlapping" counterpart of the "lacunary" result in Theorem 9.4 from \cite{DLTT}.
\begin{proposition}
\label{BMOtree}
 For each tree $\T$  with top $(I_{\T},\xi_{\T})$, each $l,M\ge 0$, $r>2$ and $1<t<\infty$
$$\|\|\sum_{s\in\T\atop{|I_s|=2^{k}}}\langle f, \varphi_s\rangle \phi_{s,\T}^{(l)}(x,\xi_{\T})\|_{V^r_k}\|_{L^t_x(\R)}\lesssim 2^{-Ml}\size(\T)|I_{\T}|^{1/t},$$
with the implicit constants depending only on $r$, $t$ and $M$.
\end{proposition}

%%%%%%%%%%%%%%%%%%%%%%%%%%%%%%%%%%%%
\section{A result on maximal multipliers}

Consider a finite set  $\Lambda=\{\lambda_1,\ldots,\lambda_N\}\subset\R$. For each $k\in \Z$ define $R_k$ to be the collection of all dyadic intervals of length $2^{k}$ containing an element from $\Lambda$.

For each $1\le r<\infty$ and each sequence $(x_k)_{k\in \Z}\in\C$, define the \emph{$r$-variational norm} of  $(x_k)_{k \in \Z}$  to be
$$\|x_k\|_{V^r_k}:=\sup_{k}|x_k|+ \|x_k\|_{\tilde{V}^r_k}$$
where
$$ \|x_k\|_{\tilde V^r_k}:= \sup_{M,\;k_0<k_1< \ldots <k_{M}}(\sum_{m=1}^M|x_{k_m}-x_{k_{m-1}}|^r)^{1/r}.$$

For each interval $\omega\in R_k$, let $m_{\omega}$ be a complex valued Schwartz function $C$-adapted to $\omega$, that is, supported on $\omega$ and satisfying
$$\|\partial^{\alpha}m_{\omega}\|_{\infty}\le C|\omega|^{-\alpha},\;\;\alpha\in\{0,1\}.$$
Define
$$
\Delta_kf(x):=\sum_{\omega\in R_k}\int m_{\omega}(\xi)\widehat{f}(\xi)e^{2\pi i\xi x}d\xi,
$$
and also
$$\|m_{\omega}\|_{V^{r,*}}:=\max_{1\le n\le N}\|\{m_{\omega_k}(\lambda_n):\lambda_n\in \omega_k\in R_k\}\|_{V^{r}_k}.$$
We will need the following result proved in \cite{Dem}.
\begin{theorem}
\label{Blemma}
For each $1<q<2$, each $\epsilon>0$, each $r>2$ and each $f\in L^q(\R)$ we have the inequality
$$\|\sup_{k}|\Delta_kf(x)|\|_{L^q_x(\R)} \lesssim N^{1/q-1/r+\epsilon}(C+\|m_{\omega}\|_{V^{r,*}})\|f\|_q,$$
with the implicit constant depending only on $r$, $\epsilon$ and $q$.
\end{theorem}

%%%%%%%%%%%%%%%%%%%%%%%%%%%%%%%%%%%%%%%%
\section{Pointwise estimates outside  exceptional sets}
\label{sec:PointEst}

Let  $\P$ be a finite convex collection of tiles which can be written as a  disjoint union of trees $\T$ with tops $T$
$$\P=\bigcup_{\T\in\F}\T.$$
To quantify better the contribution  coming from individual tiles, we need to reorganize the collection $\F$ in a more suitable way. For each  $\T\in\F$ define its saturation
$$G(\T):=\{s\in \P:\omega_T\subseteq\omega_s\}.$$
For the purpose of organizing $G(\T)$ as a collection of  disjoint and better spatially localized trees we  define for each $l\ge 0$ and $m\in\Z$ the tree
$\T_{l,m}$ to include all tiles $s\in G(\T)$ satisfying  the following requirement:
\begin{itemize}
\item $I_s\cap 2^lI_T\not=\emptyset$, if $m=0$
\item $I_s\cap (2^lI_T+2^lm|I_T|)\not=\emptyset$ and $I_s\cap (2^lI_T+2^l(m-1)|I_T|)=\emptyset$, if $m\ge 1$
\item $I_s\cap (2^lI_T+2^lm|I_T|)\not=\emptyset$ and $I_s\cap (2^lI_T+2^l(m+1)|I_T|)=\emptyset$, if $m\le -1$
\end{itemize}
We remark that since $|I_s|\le |I_T|$ for each $s\in G(\T)$, for a fixed $l\ge 0$, each $I_s$ can intersect at most two intervals $2^lI_T+2^lm|I_T|$ (and they must be adjacent).
Obviously, for each $l\ge 0$ the collection consisting of $(\T_{l,m})_{m\in\Z}$ forms a partition of $G(\T)$ into trees. The top  of $\T_{l,m}$ is formally assigned to be the pair $(I_{\T_{l,m}},\xi_{\T})$, where $I_{\T_{l,m}}$ is  the interval $(2^l+2)I_T+2^lm|I_T|$ while $\xi_{\T}$ is the frequency component of the top $(I_{\T},\xi_{\T})$ of the tree $\T$.

Denote by $\F_{l,m}$ the collection of all the trees $\T_{l,m}$.
Consider $\sigma,\gamma>0$, $\beta\ge 1$,  $r>2$ and the complex numbers $a_s,s\in\P.$

\begin{theorem}
\label{mainineq}
Let $1<q<2$. Assume we are in the settings from above and also that the following additional requirement is satisfied
\begin{equation}
\label{BMNCARseq}
\sup_{s\in\P}\frac{|a_s|}{|I_s|^{1/2}}\le \sigma.
\end{equation}
Define the  exceptional sets
\begin{align*}
E^{(1)}&:=\bigcup_{l\ge 0}\{x:\sum_{\T\in\F}1_{2^lI_T}(x)>\beta2^{2l}\},\\
E^{(2)}&:=\bigcup_{l,m\ge 0}\bigcup_{\T\in\F_{l,m}^{}}\{x:\|\sum_{s\in \T\atop{|I_s|=2^j}}a_s\phi_{s,\T}^{(\alpha(l,m))}(x,\xi_{\T})\|_{V^r_j(\Z	 )}>\gamma2^{-10l}(|m|+1)^{-2}\},
\end{align*}
where the  $\alpha(l,m)$ equals $l$ if $m\in \{-1,0,1\}$ and $l+[\log_2|m|]$ otherwise.

Then, for each $0<\epsilon<1$ (say), and for each $x\notin E^{(1)}\cup E^{(2)}$ we have
\begin{equation}\label{forex}
\|(\sum_{s\in \P\atop{|I_s|=2^k}}a_s\phi_s(x,\theta))_{k\in\Z}\|_{M_{q,\theta}^*(\R)}\lesssim \beta^{1/q-1/r+\epsilon}(\gamma+\sigma),
\end{equation}
with the implicit constants depending only on $r,\epsilon$ and $q$.
\end{theorem}
\begin{proof}
For each $l\ge 0$  and each $x\in \R$ define inductively
\begin{align*}
\F_{0,x} &:=\{\T\in\F,x\in I_T\} \\
\F_{l,x} &:=\{\T\in\F,x\in 2^lI_T\setminus 2^{l-1}I_T\},\;\;l\ge 1\\
\P_{0,x} &:=\bigcup_{\T\in \F_{0,x}}G(\T)\\
\P_{l,x}&:=\bigcup_{\T\in \F_{l,x}}G(\T)\setminus\bigcup_{l'<l}\P_{l',x},\;\;l\ge 1\\
\tilde{\F}_{l,x} &:=\{\T\in\F_{l,x}:G(\T)\setminus\bigcup_{l'<l}\P_{l',x}\not=\emptyset\}\\
\Xi_{x,l}&:=\{c(\omega_T):\;\T\in \tilde{\F}_{l,x}\}.
\end{align*}
Note that for each $x\in\R$, $\{\P_{l,x}\}_{l\ge 0}$ forms a partition of $\P$. Since $x\notin E^{(1)}$, it also follows that $\sharp \Xi_{x,l}\le \beta 2^{2l}$.

Fix  $x\not\in E^{(1)}\cup E^{(2)}\cup E^{(3)}$. Next, we fix $l$ and try to estimate
$$\|(\sum_{s\in \P_{l,x}\atop{|I_s|=2^k}}a_s\phi_s(x,\theta))_{k\in\Z}\|_{M_{q,\theta}^*(\R)}.$$

Note that for each $\lambda\in \Xi_{x,l}$,
\begin{equation}
\label{ncbdbcvhgfu90q3493-2=045934756745yjfnvfmghjrtkgh}
\{s\in \P_{l,x}: \lambda\in\omega_s\}= G(\T')\cap\P_{l,x},
\end{equation}
for some $\T'\in \tilde{\F}_{l,x}$ (and -perhaps surprisingly- if $\lambda=c(\omega_T)$, $\T'$ is not necessarily the tree  whose top tile is $T$). Indeed, let $\omega$ be the shortest frequency component of a tile $s$ from $\P_{l,x}$ such that $\lambda\in\omega$. In other words, $\omega=\omega_s$. This tile belongs to $G(\T')$, for some $\T'\in \tilde{\F}_{l,x}$ (if there are more such $\T'$, select any of them). Note that its top tile $T'$ must be in $\P_{l,x}$ (otherwise, it must be that $T'\in \bigcup_{l'<l}\P_{l',x}$, hence $T'$ was eliminated earlier, and thus the whole $G(\T)$ must have been eliminated at the same stage). \eqref{ncbdbcvhgfu90q3493-2=045934756745yjfnvfmghjrtkgh} is now immediate.

For each $\T\in \tilde{\F}_{l,x}$ define
$$\T_{l,m,x}:=\T_{l,m}\cap \P_{l,x},$$
and note that $(\T_{l,m,x})_{m}$ partition $G(\T)\cap \P_{l,x}$. An important observation is that for each $k$, the set
$\{s\in \T_{l,m,x}:|I_s|=2^k\}$ either equals $\{s\in \T_{l,m}:|I_s|=2^k\}$, or else it is empty. As a consequence,
\begin{equation}
\label{nbcdhgveyf7824t3489t0o4gmrkbmngkbjhuyq7ry}
\|\sum_{s\in\T_{l,m,x}\atop{|I_s|=2^{k}}}a_s \phi_{s,\T}^{(\alpha(l,m))}(x,\xi_{\T})\|_{V^r_k}\le \|\sum_{s\in\T_{l,m}\atop{|I_s|=2^{k}}}a_s \phi_{s,\T}^{(\alpha(l,m))}(x,\xi_{\T})\|_{V^r_k}.
\end{equation}
For each dyadic $\omega$ denote by
$$m_\omega(\theta):=\sum_{s\in \P_{l,x}:\omega_s=\omega}a_s \phi_{s}(x,\theta).$$
The key observation is that if $s\in \P_{l,x}$ and $l\ge 1$, then $x\notin 2^{l-1}I_s$, as can be easily checked. This together with
property \eqref{xbfgr1456bb} easily implies that  $m_\omega$ is $O(2^{-10l}\sigma)-$ adapted to $\omega$. Theorem \ref{Blemma} applied to $\Lambda:=\Xi_{l,x}$, and \eqref{ncbdbcvhgfu90q3493-2=045934756745yjfnvfmghjrtkgh} imply that
$$\|(\sum_{s\in \P_{l,x}\atop{|I_s|=2^k}}a_s\phi_s(x,\theta))_{k\in\Z}\|_{M_{q,\theta}^*(\R)}\lesssim 2^{4l}\beta^{1/q-1/r+\epsilon}(2^{-10l}\sigma+\max_{\T\in \tilde{\F}_{l,x} }\|\sum_{s\in\G(\T)\cap \P_{l,x}\atop{|I_s|=2^{k}}}a_s \phi_{s}(x,\xi_{\T})\|_{V^r_k}).$$

It remains to show that for each $\T\in \tilde{\F}_{l,x}$, $\|\sum_{s\in\G(\T)\cap \P_{l,x}\atop{|I_s|=2^{k}}}a_s \phi_{s}(x,\xi_{\T})\|_{V^r_k}\lesssim 2^{-10l}\gamma$. Another key observation is that if $s\in \T_{l,m,x}$ and $l\ge 1$, then $x\notin 2^{\alpha(l,m)-1}I_s$. It follows that for each $l\ge 0$ and each $m\in\Z$,
$$\phi_s(x,\theta)=\phi_{s,\T}^{(\alpha(l,m))}(x,\theta).$$
Using this and \eqref{nbcdhgveyf7824t3489t0o4gmrkbmngkbjhuyq7ry} we get that
\begin{align*}
\|\sum_{s\in\G(\T)\cap \P_{l,x}\atop{|I_s|=2^{k}}}a_s \phi_{s}(x,\xi_{\T})\|_{V^r_k}&\le \sum_{m}\|\sum_{s\in \T_{l,m,x}\atop{|I_s|=2^{k}}}a_s \phi_{s}(x,\xi_{\T})\|_{V^r_k}\\&= \sum_{m}\|\sum_{s\in \T_{l,m,x}\atop{|I_s|=2^{k}}}a_s \phi_{s,\T}^{(\alpha(m,l))}(x,\xi_{\T})\|_{V^r_k}\\&\le \sum_{m}\|\sum_{s\in \T_{l,m}\atop{|I_s|=2^{k}}}a_s \phi_{s,\T}^{(\alpha(l,m))}(x,\xi_{\T})\|_{V^r_k}
\end{align*}
Finally, since $x\notin E^{(2)}$, the last sum is $O(2^{-10l}\gamma)$, as desired. Now, \eqref{forex} follows from the triangle inequality.
\end{proof}

%%%%%%%%%%%%%%%%%%%%%%%%%%%%%%%%%%%%%%%%%%%%%%%
%%%%%%%%%%%%%%%%%%%%%%%%%%%%%%%%%%%%%%%%%%%%%%%%
\section{Proof of Theorem \ref{ceamaiceahh}}
\label{section:last}

For each collection of tiles $\S'\subseteq \S$ define the following operator
\begin{align*}
V_{\S'}f(x)&:=\|(\sum_{s\in \S'\atop{|I_s|=2^k}}\langle f,\varphi_s\rangle\phi_s(x,\theta))_{k\in\Z}\|_{M_{q,\theta}^*(\R)}.
\end{align*}

Note that for each $\S'$ the operator $V_{\S'}$ is sublinear as a function of $f$. Also, for each $f$ and $x$ the mapping  $\S'\to V_{\S'}f(x)$ is sublinear as a function of the tile set $\S'$. Let $1<p<2$ be fixed such that $\frac1p+\frac1q<\frac32$. We will prove in the following that for each $\delta>0$ and each $0<\lambda<1$
\begin{equation}
\label{en.6}
m\{x:V_{\S}1_{F}(x)\gtrsim \lambda^{1-\delta}\}\lesssim_{\delta,p,q} \frac{|F|}{\lambda^p}.
\end{equation}
Since the range of $p$ is open, this will immediately imply Theorem \ref{ceamaiceahh}. Fix now $\delta>0$. Let $\epsilon>0$ be sufficiently small, depending on $\delta$. Its value will not be specified, but it will be clear from the argument below that such an $\epsilon$ exists. Define $Q=\frac1q-\frac12+\epsilon$. We can arrange that $Q<1-\frac1p$, and define $b:=\frac{1-pQ}{1-2Q}$. It will follow that
\begin{equation}
\label{Eq:MN1}
0<b<p
\end{equation}
We can also arrange that
\begin{equation}
\label{Eq:MN2}
\epsilon+(2+\epsilon)Q<1
\end{equation}

Define the  first exceptional set
$$E:=\{x:M1_{F}(x)\ge \lambda^b\}.$$
Note that
\begin{equation}
\label{dnvjqhrfo5t8945upghi46y90856y0}
|E|\lesssim \frac{|F|}{\lambda^p}.
\end{equation}
Split $\S=\S_1\cup\S_2$ where
\begin{align*}
\S_1&:=\{s\in\S:I_s\cap E^{c}\not=\emptyset\}\\
\S_2&:=\{s\in\S:I_s\cap E^{c}=\emptyset\}.
\end{align*}

We first argue that
\begin{equation}
\label{en.11}
m\{x\in \R: V_{\S_1}1_F(x)\gtrsim \lambda^{1-\delta}\}\lesssim \frac{|F|}{\lambda^p}.
\end{equation}

Proposition ~\ref{g11a00uugg55} guarantees that $\size(\S_1)\lesssim \lambda^b$, where the size is understood here with respect to the  function $1_F$. Define $\Delta:=[-\log_2(\size(\S_1))].$ Use the result of Proposition ~\ref{Besselsineq} to split $\S_1$ as a disjoint union $\S_1=\bigcup_{n\ge \Delta}\P_n,$ where $\size(\P_n)\le 2^{-n}$ and each $\P_n$ consists of a family $\F_{\P_n}$ of trees  satisfying
\begin{equation}
\label{e:intp1}
\sum_{\T\in\F_{\P_n}}|I_T|\lesssim 2^{2n}|F|.
\end{equation}

For each $n\ge \Delta$ define $\sigma=\sigma_n:=2^{-n}$, $\beta=\beta_n:=2^{(2+\epsilon)n}\lambda^{p}$, $\gamma=\gamma_n:=2^{-n[(2+\epsilon)Q+\epsilon]}\lambda^{1-Qp-3\epsilon}$. Define $a_s:=\langle 1_F,\varphi_s\rangle$ for each $s\in\P_n$ and note that the collection $\P_n$ together with the coefficients $(a_s)_{s\in\P_n}$ satisfy the requirements of Theorem ~\ref{mainineq}. Let $\F_{\P_n,l,m}$ be the collection of all the trees $\T_{l,m}$ obtained from all the trees $\T\in\F_{\P_n}$ by the procedure described in the beginning of the previous section. Let $r>2$ be any number such that $\frac1q-\frac1r<Q$. Define the corresponding exceptional sets
\begin{align*}
E^{(1)}_{n}&:=\bigcup_{l\ge 0}\{x:\sum_{\T\in\F_{\P_n}}1_{2^lI_T}(x)>\beta_n2^{2l}\},\\
E^{(2)}_{n}&:=\bigcup_{l,m\ge 0}\bigcup_{\T\in\F_{\P_n,l,m}}\{x:\|\sum_{s\in \T\atop{|I_s|<2^j}}a_s\phi_{s,\T}^{(\alpha(l,m))}(x,\xi_{\T})\|_{V^r_j(\Z)}>\gamma_n2^{-10l}(|m|+1)^{-2}\}.
\end{align*}
By  ~\eqref{e:intp1}  we get
$$|E^{(1)}_{n}|\lesssim 2^{-n\epsilon}\lambda^{-p}|F|.$$
By Theorem ~\ref{BMOtree}, for each $1<s<\infty$ we get
$$|E^{(2)}_{n}|\lesssim \gamma_n^{-s}\sigma_n^{s-2}|F|\lesssim 2^{-n[-2-s(-1+(2+\epsilon)Q+\epsilon)]}\lambda^{-s(1-Qp-3\epsilon)}|F|. $$
Define
$$E^{*}:=\bigcup_{n\ge \Delta}(E^{(1)}_{n}\cup E^{(2)}_{n}).$$ Trivial computations show that since $\lambda\le 1$ and since $2^{-\Delta}\lesssim \lambda^b$, we have $|E^{*}|\lesssim \lambda^{-p}|F|,$ (work with  a sufficiently large $s$, depending only on $p,q,Q,\epsilon$).

For each $x\notin E^{*}$, Theorem ~\ref{mainineq} guarantees that
$$
 V_{\S_1}1_F(x)\le \sum_{n\ge \Delta}V_{\P_n}1_F(x)\lesssim \sum_{n\ge \Delta}\beta_n^Q(\gamma_n+\sigma_n).
$$
The latter sum is easily seen to be $O(\lambda^{1-\delta})$, if $\epsilon$ is sufficiently small.
This ends the proof of ~\eqref{en.11}. We next prove that (and note that this is enough, due to \eqref{dnvjqhrfo5t8945upghi46y90856y0})
\begin{equation}
\label{en.11hsdgjhgsdytweyehcjxzkJSkcsp'aPosQIS}
m\{x\notin E: V_{\S_2}1_F(x)\gtrsim \lambda^{1-\delta}\}\lesssim \frac{|F|}{\lambda^p}.
\end{equation}

To achieve this, we split
$$\S_2:=\bigcup_{\kappa>0}\S_{2,\kappa},$$
where
$$\S_{2,\kappa}:=\{s\in\S_2:2^{\kappa-1}I_s\cap E^{c}=\emptyset,\;2^{\kappa}I_s\cap E^{c}\not=\emptyset\},$$
and we prove that, uniformly over $\kappa>0$,
\begin{equation}
\label{bcvnsdbvghrgfyqye903892-r4i34opfregnrtkbmgl;hnkety0960-oy76ijoykj}
m\{x\notin E: V_{\S_{2,\kappa}}1_F(x)\gtrsim 2^{-\kappa}\lambda^{1-\delta}\}\lesssim 2^{-\kappa}\frac{|F|}{\lambda^p}.
\end{equation}
Note further that if $s\in\S_{2,\kappa}$ then
$$
\frac{|\langle 1_F,\varphi_s\rangle|}{|I_s|^{1/2}}\lesssim \inf_{x\in I_s}\M1_F(x)\lesssim 2^{\kappa}\inf_{x\in 2^{\kappa}I_s}\M1_F(x)\lesssim\lambda^b2^{\kappa},
$$
and thus $\S_{2,\kappa}$ has size $O(\lambda^b2^{\kappa})$. Note also that $\S_{2,\kappa}$ remains convex. The proof of \eqref{bcvnsdbvghrgfyqye903892-r4i34opfregnrtkbmgl;hnkety0960-oy76ijoykj} now follows exactly the same way as the proof of ~\eqref{en.11}. The fact that the size of $\S_{2,\kappa}$ is (potentially) greater than that of $\S_1$ by a factor of $2^{\kappa}$ is compensated by the fact that for each $x\notin E$ and each $s\in \S_{2,\kappa}$, $x\notin 2^{\kappa-1}I_s$. It follows that in the definition of the exceptional sets $E_n^{(1)}$ and $E_n^{(2)}$ for this case, we can can restrict the union to $l\ge \kappa-1$. We leave the details to the interested reader.

%%%%%%%%%%%%%%%%%%%%%%%%%%%%%%

\end{document}